\documentclass{article}

\usepackage{amsfonts}
\usepackage{amssymb}
\usepackage{amsthm}
\usepackage{amsmath}
\usepackage{epsfig}
\usepackage{psfrag}
\usepackage{url}
\usepackage{mathdots}
\usepackage{xcolor}

\newtheorem{theorem}{Theorem}[section]
\newtheorem{corollary}[theorem]{Corollary}
\newtheorem{lemma}[theorem]{Lemma}

\newtheorem{definition}[theorem]{Definition}

\newtheorem{*theorem}[theorem]{*Theorem}
\newtheorem{*corollary}[theorem]{*Corollary}
\newtheorem{*lemma}[theorem]{*Lemma}
\newtheorem{*proposition}[theorem]{*Proposition}

\def\Sph{{S^2}}

\def\real{\mathbb{R}}

\def\vertices{\operatorname{vertices}}

\begin{document}

\title{Embedding a pair of graphs in a surface, and the width of $4$-dimensional prismatoids}

\author{
Francisco Santos\thanks{Supported in part by the Spanish Ministry of Science through grant MTM2008-04699-C03-02},
Tamon Stephen\thanks{Supported in part by an NSERC Discovery Grant.}, and
Hugh Thomas\thanks{Supported in part by an NSERC Discovery Grant.}}
\date{}

\maketitle

\begin{abstract}
A \emph{prismatoid} is a polytope with all its vertices contained in two parallel facets, called its \emph{bases}. Its \emph{width} is the number of steps needed to go from one base to the other in the dual graph. The first author recently showed that the existence of counter-examples to the Hirsch conjecture is equivalent to that of $d$-prismatoids of width larger than $d$, and constructed such prismatoids in dimension five. Here we show that the same is impossible in dimension four. This is proved by looking at the pair of graph embeddings on a $2$-sphere that arise from the normal fans of the two bases of $Q$.
\end{abstract}

\section{Prismatoids and pairs of maps}

The first author recently constructed counter-examples to the Hirsch conjecture~\cite{Santos:Hirsch-counter} via the combination of the following two results:
\begin{itemize}
\item If a $d$-dimensional prismatoid $Q$ has width larger than $d$ then it is possible to construct from it a polytope $P$ that violates the Hirsch conjecture.
\item There exist $5$-prismatoids of width larger than 5.
\end{itemize}

Here, a \emph{prismatoid} is a polytope that has all its vertices lying in two parallel facets $Q^+$ and $Q^-$ (called its \emph{bases}), and its \emph{width} is the dual graph distance between those two facets. That is, the width of $Q$ is the minimum number of steps needed to go from $Q^+$ to $Q^-$, where a step consists in moving from a facet of $Q$ to an adjacent facet.

The dimension of the non-Hirsch polytope $P$ in the construction is much higher than that of $Q$ (it equals $|\vertices(Q)| - |\dim(Q)|$) but still the question arose whether the dimension $d=5$ for $Q$ is the smallest possible. An easy argument shows $d=3$ cannot work, and $d=4$ was left open in~\cite{Santos:Hirsch-counter}. In this paper we answer this question in the negative:

\begin{theorem}
\label{thm:4prismatoid}
The width of a $4$-dimensional prismatoid is at most four.
\end{theorem}

The proof of this theorem uses the following reduction, also from~\cite{Santos:Hirsch-counter}. Let $Q\subset\real^d$ be a prismatoid with bases $Q^+$ and $Q^-$. Consider the bases as simultaneously embedded into $\real^{d-1}$ and let $G^+$ and $G^-$ be the intersections of their normal fans with the unit sphere $S^{d-2}$. $G^+$ and $G^-$ are geodesic cell decompositions of $S^{d-2}$, what we call \emph{geodesic maps}. Let $H$ be their common refinement: cells of $H$ are all the intersections of a cell of $G^+$ and a cell of $G^-$. Then:
\begin{itemize}
\item All facets of $Q$ other than the two bases appear as vertices of $H$.
\item The facets adjacent to $Q^+$ (respectively to $Q^-$) appear in $H$ as the vertices of $G^+$ (respectively of $G^-$).
\item Adjacent facets of $Q$ appear as vertices connected by an edge of $H$.
\end{itemize}

As a consequence, we get the following result, in which we call \emph{width} of the pair of geodesic maps $(G^+, G^-)$ the minimum graph distance  along $H$ from a vertex of $G^+$ to a vertex of $G^-$:

\begin{lemma}
\label{lemma:maps}
The width of a prismatoid $Q\subset\real^d$ equals two plus the width of its corresponding pair of maps $(G^+,G^-)$ in $S^{d-2}$.
\end{lemma}

In particular, the width of a $4$-prismatoid is two plus the width of a pair of maps in the $2$-sphere. In the $2$-sphere, knowing the (embedded) graph of a map is enough to recover the whole map, so we rather speak of a \emph{pair of graphs} embedded in it.
If the two graphs intersect other than in the interior of edges, then the width of the pair is one or zero, so in what follows we assume that this does not happen. Moreover, since the maps we are interested in are geodesic, an edge of one cannot intersect an edge of the other more than once, and they do so transversally; that is, the four branches that come out from an intersection point alternate between $G^+$ and $G^-$ when we look at them cyclically. Putting together these conditions, we introduce the following notion, which looks natural in the context of geometric graph theory:

\begin{definition}
\label{def:normal-pair} Let $S$ be a closed surface and let $G^+$ and $G^-$ be two graphs embedded in it. We say that $(G^+,G^-)$ is a
\emph{normal pair of graphs} if $G^+$ and $G^-$ intersect only in the interior of edges,
each edge of one intersects each edge of the other at most once, and all the intersections are transversal.
\end{definition}

We also consider a somewhat stronger condition:

\begin{definition}
We say that $(G^+,G^-)$ form a \emph{strongly normal pair of graphs} if they
are a normal pair of graphs and in addition the intersection of an edge
of $G^+$ with the interior of a face of $G^-$ has at most one component (and
also with the roles of $G^+$ and $G^-$ interchanged).  
\end{definition}

This condition is also satisfied by any pair of geodesic maps, since it 
amounts to saying that an edge of $G^+$ cannot leave the interior of a geodesic
cell of $G^-$ and then re-enter it, which follows from convexity.  


We call \emph{width}\ of the normal pair the minimum distance from a vertex of $G^+$ to one of $G^-$ along the common refinement of the two graphs. Put differently, the width is one plus the minimum number of points of $G^+\cap G^-$ that you need to go through in order to get from a vertex of $G^+$ to one of $G^-$.  We assume that $G^+$ and $G^-$ intersect at least once. (Otherwise the width can be considered infinite, but the question we address in this paper is meaningless). 

As an example, the two graphs in Figure~\ref{fig:2d-grid} form a normal pair of width five in the torus, and the same idea shows the existence of normal pairs of arbitrarily large width in any closed surface of non-positive Euler characteristic. This example is taken from~\cite{5prismatoids}, where it is used to show the existence of prismatoids \emph{of dimension five} of arbitrarily large width. Our main result, which implies Theorem~\ref{thm:4prismatoid}, is that such pairs do not exist in the $2$-sphere:

\begin{figure}
\centerline{
\includegraphics[scale=.37]{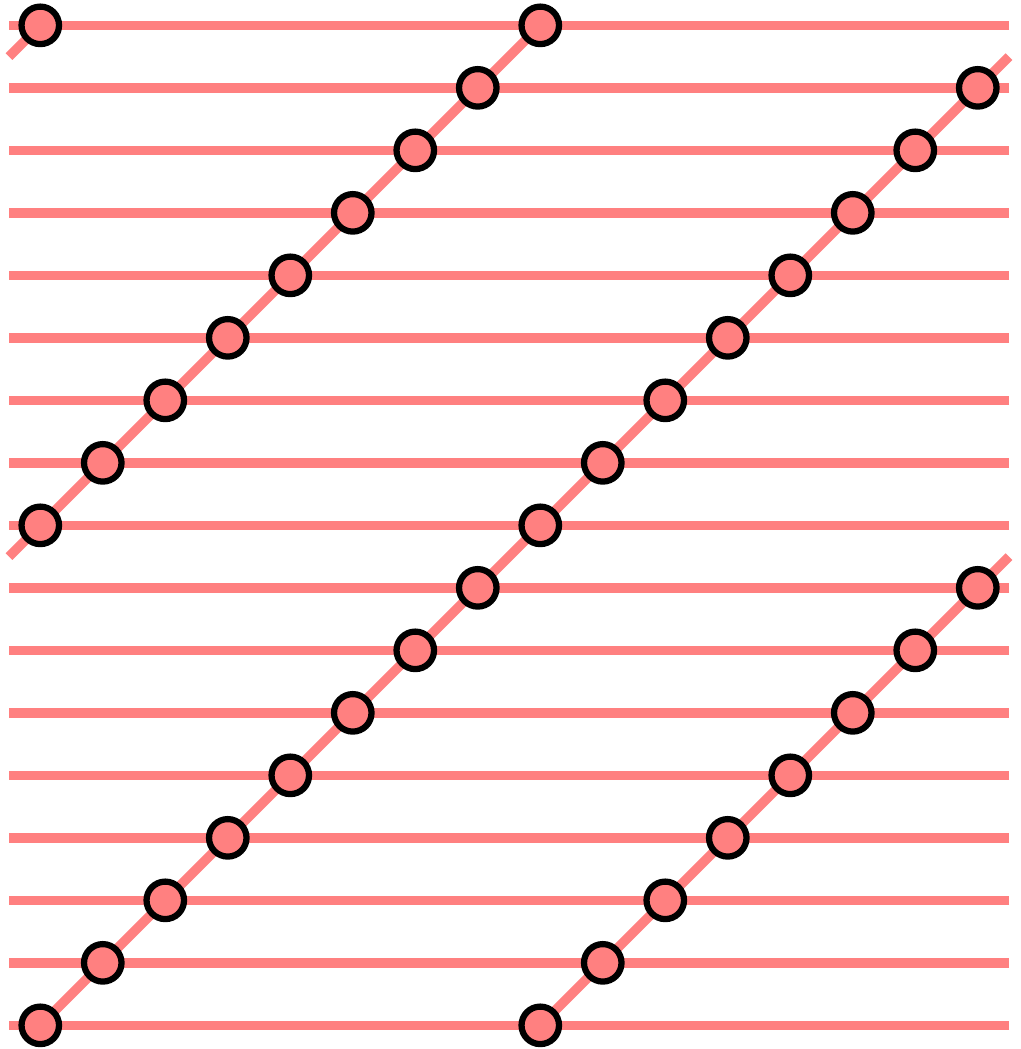}\ 
\includegraphics[scale=.37]{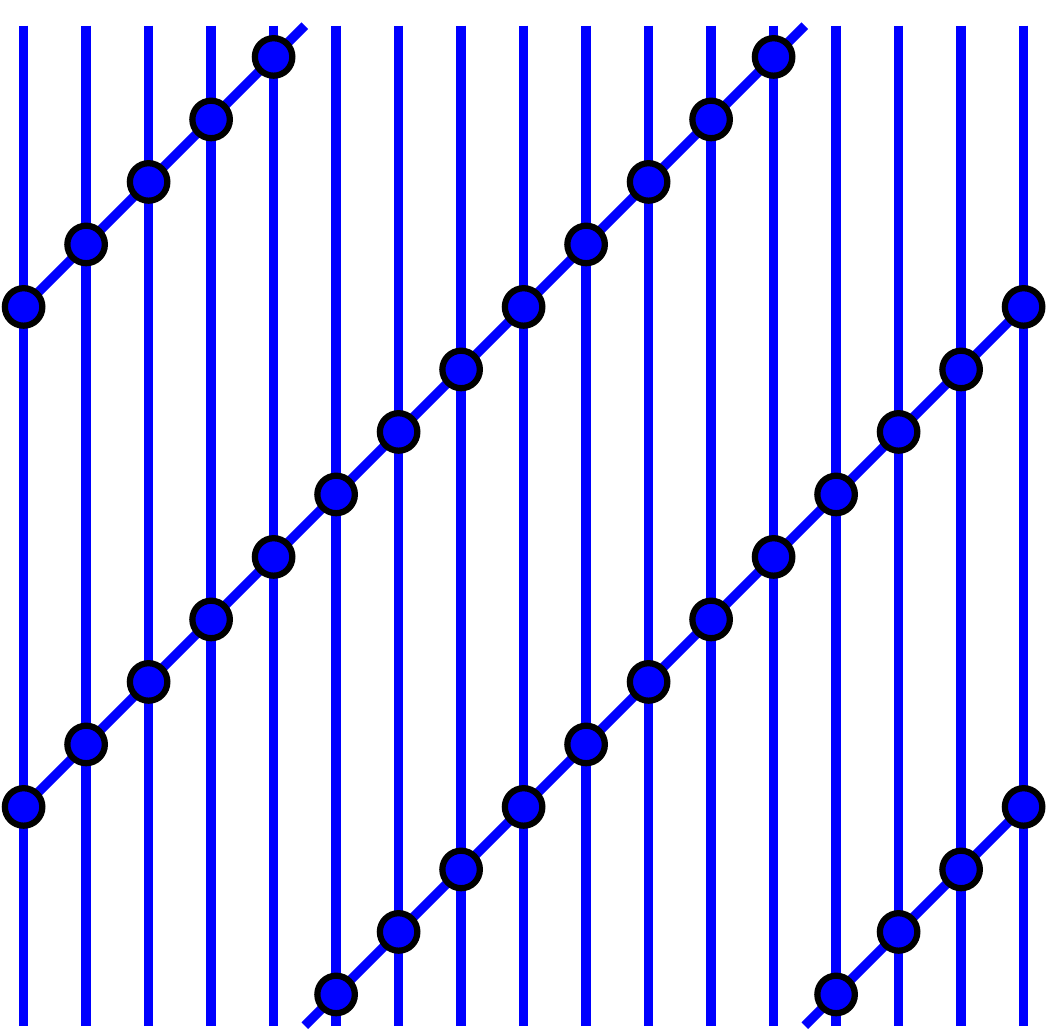}\ 
\includegraphics[scale=.37]{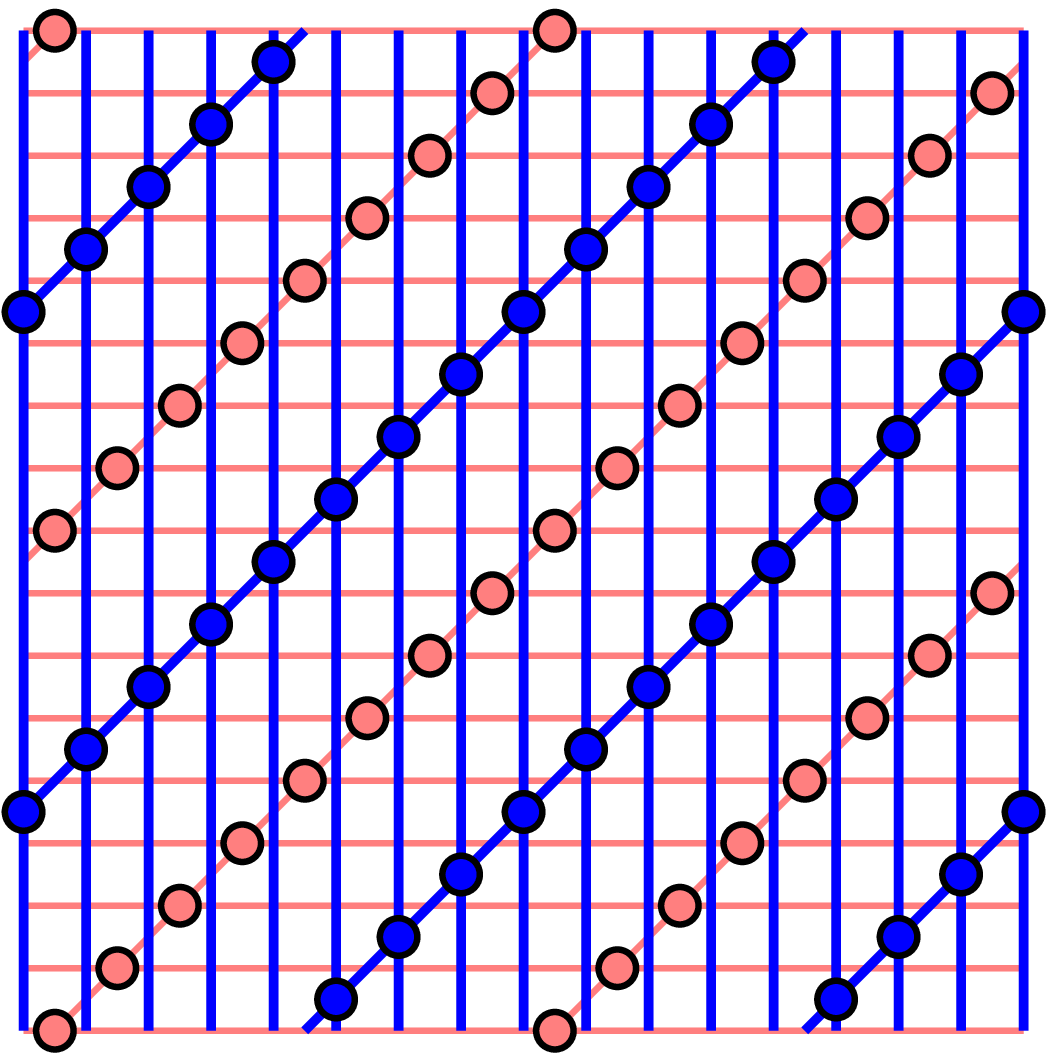}
}
\caption{Two graphs embedded in the torus (left and center) and the normal map they form (right), of width five}
\label{fig:2d-grid}
\end{figure}

\begin{theorem}
\label{thm:width}
Every normal pair of graphs in the sphere or the projective plane has width two.
\end{theorem}

We prove this theorem in Section 3, using an argument based on the 
simply-connectedness of the sphere.  
In Section 4, we use the notion of Euler characteristic to provide an
independent proof of the theorem, under the stronger assumption that
the pair of graphs is strongly normal (which is still sufficient to 
deduce Theorem \ref{thm:4prismatoid}).


\section{Terminology and initial observations} \label{initial}

Start with an arbitrary normal pair of graphs $(G^+,G^-)$ in an arbitrary closed surface $S$. We introduce the following nomenclature: we call $G^+$ and $G^-$ the \emph{positive} and \emph{negative} graphs, and let $H$ be the common graph induced by the embedding. In $H$ we have three types of vertices: \emph{positive} (vertices of $G^+$), \emph{negative} (vertices of $G^-$), and \emph{crossings}. We also have four types of edges: each edge is \emph{positive} or \emph{negative} depending on whether it is contained in an edge of $G^+$ or one of $G^-$, and it is \emph{deep} or \emph{terminal} depending on whether its two end-points are crossings or at least one of them is not. Observe that edges ending in no crossings (that is, edges of $G^+$ or $G^-$ that do not cross the other graph of the pair) do not contribute to the width, so we may as well contract them and assume our pair of graphs does not have any. Contraction may make our graphs be not simple, so we admit that possibility from the beginning.

\begin{figure}
\centerline{
\includegraphics[scale=.5]{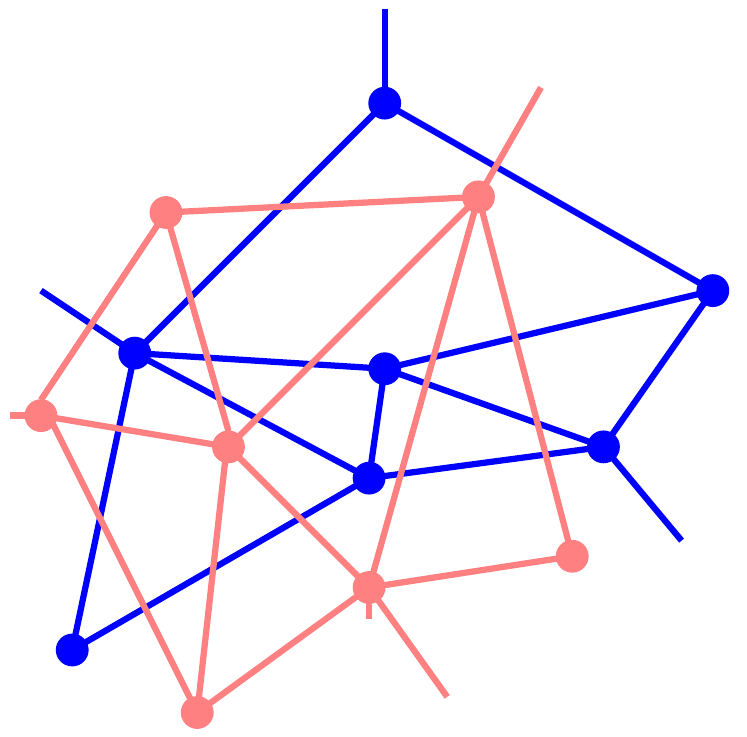}\ 
\includegraphics[scale=.5]{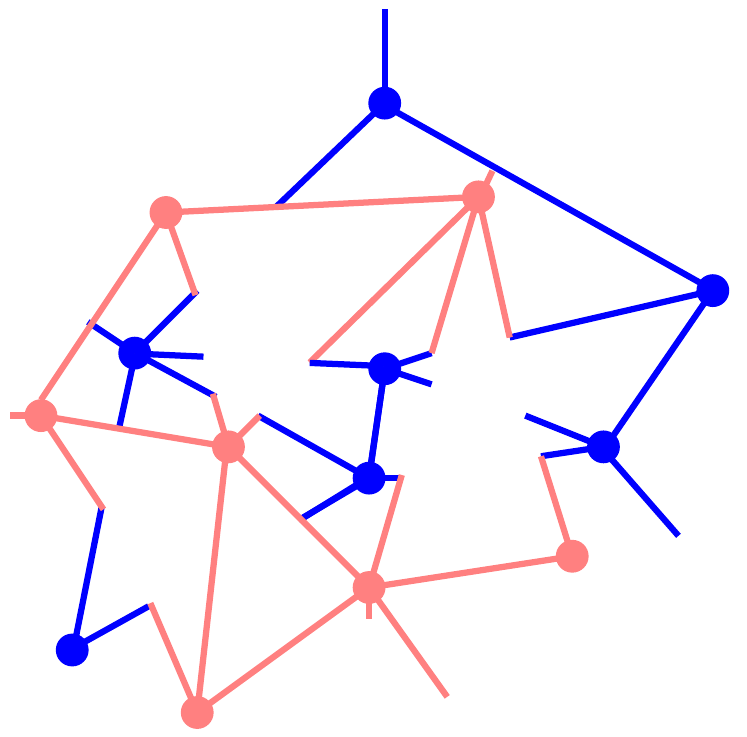}\ 
\includegraphics[scale=.5]{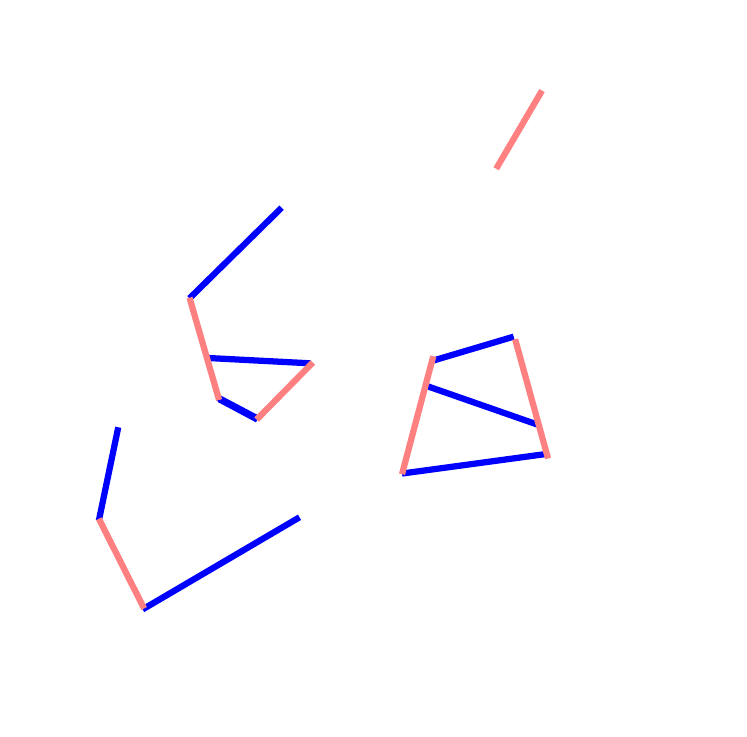}
}\caption{A normal pair of graphs (left), its terminal part (center) and its deep part (right)}
\label{fig:deep-terminal}
\end{figure}

We also show that there is no loss of generality in assuming that $S$ is 
orientable.  

\begin{lemma}
\label{lemma:oriented}
If $(G^+,G^-)$ is a normal pair in a non-orientable surface $S$, then the pull-back of $(G^+,G^-)$ via the double cover $S'\to S$ from an orientable surface $S'$ yields a normal pair  $({G^+}',{G^-}')$ in $S'$ with the same width.
\end{lemma}

\begin{proof}
Every path from $G^+$ to $G^-$ in $S$ lifts to a path (actually to two paths) from ${G^+}'$ to ${G^-}'$ in $S'$, of the same length. Reciprocally, every path in $S'$ projects to $S$.
\end{proof}

%

\section{First proof of Theorem~\ref{thm:width}} 

As in the previous section, start with an arbitrary normal pair of  
graphs $(G^+,G^-)$ in an arbitrary closed surface $S$.
This proof concentrates on the \emph{deep} subgraph of $H$, that we denote $H_0$. This is the subgraph consisting of deep edges or, equivalently, the subgraph induced by the crossing points. Observe that if $H_0$ has no edges then the width of the pair $(G^+,G^-)$ is trivially two: every terminal positive edge is adjacent to a terminal negative edge, and those two edges from a path from $G^+$ to $G^-$ of length two. More generally:

\begin{lemma}
\label{lemma:degree3}
The width of the pair $(G^+,G^-)$ exceeds two if, and only if, every vertex of $H_0$ is incident to either two deep positive edges or two deep negative edges.
\end{lemma}

\begin{proof}
The width equals two if, and only if, there is a crossing incident to both a terminal positive and a terminal negative edge.
If this does not occur then every crossing is incident to two deep edges of one sign (and to zero, one or two of the other).
\end{proof}

By Lemma \ref{lemma:oriented}, there is no loss of generality in assuming
that $S$ is orientable.  We therefore assume that $S$ is orientable, and
give an orientation to it. In particular, this gives a meaning to the
expressions ``left'' and ``right'' of a directed edge.  

\begin{definition}
\label{def:loop}
\rm
A path in $H_0$ is \emph{well-formed} if when traversing $C$ we always turn right when passing from a negative to a positive edge and turn left when passing from a positive to a negative edge. (The path is allowed to ``go straight'' at some of the vertices, meaning that it keeps walking along the same edge of $G^+$ or $G^-$). A \emph{well-formed loop} is a cycle with the same property except 
(perhaps) at its base vertex.
\end{definition}


Observe that the definition is independent of the direction in which we traverse the path or cycle, since reversing the direction exchanges left and right turns but it also exchanges passings from positive to negative and passings from negative to positive. The following two lemmas prove Theorem~\ref{thm:width}.

\begin{lemma}
\label{lemma:loop}
If a normal pair of graphs has width larger than two then it contains a well-formed loop.
\end{lemma}

\begin{proof}
Starting at any deep edge $e$ we can build a well-formed path as follows: go straight on edges of the same sign as $e$ until you cannot go further. When this happens, Lemma~\ref{lemma:degree3} guarantees that you are at a vertex incident to two edges of the other sign, one to the left and one to the right. Choose the one that makes the path be well-formed and continue.

Since this process can always be continued, at some point the path will cross with itself. The first time this happens, you declare the repeated vertex to be the base for the loop obtained, which is well-formed by construction. (Observe that the original edge $e$ may not belong to the loop).
\end{proof}

\begin{lemma}
\label{lemma:disc}
A well-formed loop in a normal pair of graphs cannot bound a disc.
\end{lemma}

\begin{proof}
Let $L$ be a well-formed loop with base vertex $v$ and suppose, to get a contradiction, that it bounds a disc $D$. We are going to show how to construct another well-formed loop in $D$, which necessarily should bound a smaller disc $D'\subset D$. Since the same applies to $D'$, we can construct an infinite sequence of loops bounding smaller and smaller discs, a contradiction.

To construct $L'$ observe that, by normality of the pair of maps, the loop $L$ has at least two turns, apart from the possible turn at the base point $v$. Since left and right turns alternate, at least one of the turns is to the left when we traverse $L$ in the clockwise direction (clockwise is understood with respect to $D$ and to the given orientation in it). Let $u$ be the vertex where this happens. Put differently, $u$ is a ``concave'' or ``reflex'' vertex in the boundary of $D$.

By Lemma~\ref{lemma:degree3}, apart from the positive and negative edges incident to $u$ in $L$, $H_0$ has at least one more deep edge
 $e$ coming out from $u$, necessarily towards 
the interior of $D$. We use that edge $e$ to start a well-formed path, as in the proof of Lemma~\ref{lemma:loop}, and stop when the path intersects 
either itself or the loop $L$. If the path intersects itself then we have found a well-formed loop contained in $D$. If it intersects $L$, then let $w$ be the 
point where this happens and let $C$ be the well-formed path that we have obtained from $u$ to $w$ through the interior of $D$. Let $C_1$ and 
$C_2$ be the two paths in which $L$ is divided by the vertices $u$ and $w$ (see Figure~\ref{fig:disc}). Our claim is that $C$ together with one of $C_1$ or $C_2$ is a well-formed loop. 

\begin{figure}[htb]
\begin{center}
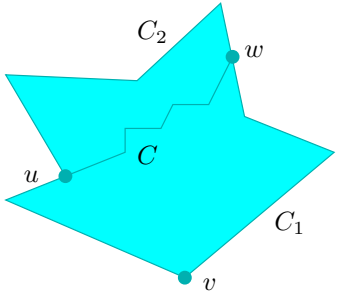
\caption{A well-formed path (left) and a well-formed loop (right)}
\label{fig:disc}
\end{center}
\end{figure}

Indeed, assume first that $v=w$. Then we start our loop from $v$ using the interior path $C$ and when we arrive at $u$ one of the turns to $C_1$ or $C_2$ is well-formed, which closes the loop. If $v$ is different from both $u$ and $w$, assume that $v$ lies in $C_1$. The key idea is that at both  
$u$ and $w$ at least one of the turns, from $C$ to $C_1$ or from $C$ to $C_2$, is valid for a well-formed path. Then:
\begin{enumerate}
\item If both turns from $C$ to $C_1$ are valid, then $C$ and $C_1$ give a well-formed loop based at $v$.

\item If the turn from $C$ to $C_1$ is invalid at, say, $u$, then $C$ and $C_2$ give a well-formed loop based at $w$, because the turn from $C$ to $C_2$ is valid at $u$.
\end{enumerate}
\end{proof}

\begin{proof}[Proof of Theorem~\ref{thm:width}]
Lemmas~\ref{lemma:loop} and~\ref{lemma:disc} imply that a normal pair of graphs of width larger than two has a loop that does not bound a disc, which cannot happen in the $2$-sphere. For the projective plane we use the reduction of Lemma~\ref{lemma:oriented}.
\end{proof}

\section{Alternative proof of Theorem~\ref{thm:width}} 

In this section, we provide an alternative proof of Theorem \ref{thm:width},
under the stronger hypothesis that the pair $(G^+,G^-)$ are strongly normal.

As discussed in Section \ref{initial}, to prove Theorem~\ref{thm:width}, 
it suffices
to assume that $(G^+,G^-)$ are a pair of graphs embedded in a sphere
$\Sph$, such that each edge of $G^+$ crosses some edge of $G^-$, and
vice versa.  Note that the reductions to this case preserve the 
property that the pair of graphs is strongly normal.
We again study the graph $H$ induced by the embedding.
In this case, we focus on the complex of 0-, 1- and 2-dimensional
cells (vertices, edges and faces) defined by the embedding. 

We assume that $H$ is connected, since it suffices to
consider some component of $H$.  The connectedness of $H$ implies that the 
faces defined by $H$ are contractible.  

We will calculate the Euler characteristic of the sphere $\Sph$, 
using the decomposition induced by $H$.  Since the faces of the decomposition are contractible,
this is just the number of vertices of $H$, minus
the number of edges, plus the number of faces. 
Our strategy will be to bound this sum by a contribution for each face.  
Assuming that the width of $(G^+,G^-)$ is greater that two, 
we will show that the contribution for each face is  
non-positive, and thus that the Euler characteristic of $\Sph$
is non-positive.  Since the Euler
characteristic of $\Sph$ is positive, this is a contradiction.  

We now discuss how to bound the Euler characteristic by a contribution 
for each face.  
Let
$F$ be a face of $H$.  Consider (as a first step) the following sum:

\begin{align*}c(F)=&\frac14 \textrm{ (the number of crossing vertices incident to $F$)} \\ &-\frac12 \textrm{ (the number of edges of $F$)}\\ & + 1.
\end{align*}
Note that, in this definition, we count edges and vertices with 
multiplicity.  That is to say, we count them going around the boundary
of the face, so that, if we revisit an edge or vertex, we count it again.  

We have the following:

\begin{lemma} Let $(G^+,G^-)$ be a normal pair of graphs such that $H$ is
connected.  Then:
$\chi(\Sph)=|V(G^+)|+|V(G^-)|+\sum_F c(F)$. \end{lemma}
\begin{proof} Each of the crossing vertices 
contributes $+1$ to the Euler characteristic
which is shared among the four faces that meet at it, 
and each edge contributes $-1$
to the Euler characteristic, which is shared between the two faces it lies on.
\end{proof}

Since we want to bound $\chi(\Sph)$ by a sum over the faces, we must decide,
for each signed vertex, where to include it in the sum.  

Suppose $v$ is a positive vertex on face $F$, and that the two crossing 
vertices of $F$ adjacent to $v$ on $F$ are $x$ and $y$.  
We say that $v$ {\it counts against}  
 $F$ if $x$ and $y$ are not incident to the same edge of $G^-$.
We make the similar definitions for negative vertices.  

\begin{lemma} \label{le:countstwice}
If $(G^+,G^-)$ is a strongly normal pair of graphs, then 
every signed vertex $v$ counts against at least two faces.  
\end{lemma}

\begin{proof} 
Suppose $v$ is positive.  There must be at least two different edges of
$G^-$ which are incident to crossing vertices of $H$ at distance one from
$v$, since otherwise the single edge of $G^-$ which is incident to all the
crossing vertices at distance one from $v$ would necessarily leave and
re-enter the interior of the same $G^+$-face, contrary to our assumption
of strong normality. 
Therefore, as we cyclically traverse the
crossing vertices at distance one from $v$, which edge from $G^-$ they 
intersect must change at least twice, so $v$ counts against at least two
faces of $H$.  \end{proof}

Let 
$d(F)=c(F)+\frac 12 \textrm{ (the number of signed vertices
counting against $F$)}$.
Given the previous lemma, the following lemma is immediate:

\begin{corollary} If $(G^+,G^-)$ is a strongly normal pair of graphs, with 
$H$ connected, then 
$\chi(\Sph) \leq \sum_F d(F)$.  \end{corollary}

We finish by showing that $d(F)\leq 0$ for all $F$, under
the assumption that the width of $(G^+,G^-)$ is greater than two.  

\begin{lemma}\label{nospec} Assume that 
$(G^+,G^-)$ is a normal pair of graphs and 
has width greater than two.  Let $F$ be a face of $H$.
Then $d(F)\leq 0$.  \end{lemma}

\begin{proof}
Write $e$ for the number of edges of $F$.  
Write $x$ for the number of crossing vertices.
Now $d(F)\leq 1+x/4-e/2+(e-x)/2=(4-x)/4$.      
We remark that $x$ counts the number of times that the sign of
the edge changes as we travel around the boundary of $F$
and thus must be even.

It is clear from the above 
that if $F$ has four or more crossing vertices then we are done.
If $F$ has two or more signed vertices, then it must have four or more
crossing vertices.  (Signed vertices cannot be at distance one, and since
the signs of edges alternate, for signed vertices to be at distance two,
they would have to have opposite signs, and thus we would have a path
of length two.)

If $F$ has no signed vertices, then again it must have at least four crossing
vertices (since the signs of the edges alternate, there are an even
number of edges, and two edges would mean that the corresponding edges of
$G^+$ and $G^-$ cross twice).  

If $F$ has one signed vertex $v$, it may only have two crossing vertices.
In this case, however, $F$ is a triangle, and $v$ doesn't count
against $F$, so $d(F)=c(F)=0$.
\end{proof}


\end{document}